\documentclass[a4paper]{article}
\usepackage{amsmath}
\usepackage{fullpage}
\usepackage{amssymb}
\usepackage{amsthm}

\usepackage{bbm}
\usepackage{mathtools}
\usepackage{color}
\usepackage{hyperref}
\usepackage{cite}
\usepackage{graphicx}
\usepackage{epstopdf}
\allowdisplaybreaks[3]
\DeclarePairedDelimiter\abs{\lvert}{\rvert}
\newtheorem{theorem}{Theorem}
\newtheorem{lemma}{Lemma}

\newtheorem{corollary}{Corollary}

\newtheorem{remark}{Remark}
\newcommand*{\bfrac}[2]{\genfrac{}{}{0pt}{}{#1}{#2}}
\newcommand{\bs}{\boldsymbol}
\title{Universality in a Class of Fragmentation-Coalescence Processes}

\author{A. E. Kyprianou\thanks{Department of Mathematical Sciences, University of Bath, Claverton Down, Bath, BA2 7AY, UK. Email: \texttt{a.kyprianou@bath.ac.uk}, \texttt{s.pagett@bath.ac.uk}, \texttt{t.c.rogers@bath.ac.uk}}, \,
S. W. Pagett$^{*}$, and T. Rogers$^{*}$  }

\begin{document}
\maketitle

\begin{abstract}
We introduce and analyse a class of fragmentation-coalescence processes defined on finite systems of particles organised into clusters. Coalescent events merge multiple clusters simultaneously to form a single larger cluster, while fragmentation breaks up a cluster into a collection of singletons. Under mild conditions on the coalescence rates, we show that the distribution of cluster sizes becomes non-random in the thermodynamic limit. Moreover, we discover that in the limit of small fragmentation rate these processes exhibit a universal cluster size distribution regardless of the details of the rates, following a power law with exponent $3/2$.
\end{abstract}

\section{Introduction}
Processes of coalescence and its reverse, fragmentation, have been widely studied in physical chemistry since the seminal work of Smoluchowski \cite{Smoluchowski1916} a century ago (see \textit{e.g.} \cite{Aldous1999,Koutzenogii1996,Seinfeld1986} and references therein). Aside from chemical systems, these processes also have an important role to play in modelling genealogy \cite{Tavare1984,Lambert2005} and even the dynamics of group formation \cite{Gueron1995}. The microscopic specification of a fragmentation-coalescence process is usually given in terms of a stochastic process acting on a finite number of constituent particles, introduced by Marcus \cite{Marcus1968}, Lushnikov \cite{Lushnikov1978a,Lushnikov1978b} and Gueron \cite{Gueron1998}. However, in applied work is it common to write deterministic `mean-field' equations that are intended to describe behaviour in thermodynamic limit of large system size. An important question arises: are these processes self-averaging so that mean-field calculations are relevant \cite{EHSS1985,Jeon1998}? 

Another key concept in the understanding of large-scale interacting systems is that of universality --- that certain important macroscopic properties often do not depend on the detailed features of the particles and dynamics involved, but rather a much smaller set of properties determine how these processes behave in the thermodynamic limit. This is particularly true of critical phenomena, a famous example being the directed percolation phase transistion \cite{Hinrichsen2000}. In fragmentation-coalescence processes, the antagonistic nature of the driving mechanisms can give rise to self-organised criticality in certain limits, whereby the system evolves to a stationary distribution that exhibits scale-invariant behaviour normally characteristic of a phase transition \cite{RathToth2009}. It is therefore natural to ask if this behaviour is universal. 

Here we answer both of the above questions in the affirmative. We examine a class of fragmentation-coalescence processes that allow simultaneous coagulation of multiple clusters with different rates, combined with independent fragmentation events that shatter clusters to singletons. The fragmentation and coalescence rates are taken to be independent of cluster size, and we make a mild technical assumption on the coalescence rates. We show that processes of this class approach a non-random limit as the system size grows, and moreover, that in the limit of small fragmentation rate we observe a universal cluster size distribution with a power-law tail of exponent $3/2$.

\newpage
\subsection{Finite Fragmentation-Coalescence Processes}\label{class}
Consider a collection of $n$ identical particles, grouped together into some number of clusters. We define a stochastic dynamical process as follows: 
\begin{enumerate}
\item Every subset of $k$ clusters coalesces at rate $\alpha(k)n^{1-k}$, independently of everything else that happens in the system. The coalescing clusters are merged to form a single cluster with size equal to the sum of the sizes of the merged clusters.
\item Clusters fragment at rate $\lambda_n>0$, independently of everything else that happens in the system. Fragmentation of a cluster of size $k$ results in $k$ `singleton' clusters of size one.
\end{enumerate}
The standard choice of initial condition is the state with $n$ singleton clusters and this will be the case throughout this paper. The factor of $n^{1-k}$ appearing in the coalescence rates is included to compensate for the combinatorial explosion in the number of subsets of $k$ clusters as $n$ gets larger. In this way, when there are order $n$ clusters, the global rates of fragmentation and coalescence of any number of clusters are all order $n$ (in the case where $\lambda_n\equiv\lambda\in(0,\infty)$). Note that this choice is necessary to ensure that there is a single dominant timescale for the dynamics. 

For each $n\in\mathbb{N}$, and $k\in\{1,\ldots,n\}$, the state of the system is specified by the number of clusters of size $k$ at time $t$. To that end we introduce the random variables
\begin{equation*}
w_{n,k}(t):=\frac{1}{n}\#\{\text{clusters of size $k$ at time $t$}\},\quad 1\leq k \leq n,
\end{equation*}
which we take to be continuous from the left, with right limits, and combine these into the random vector 
\begin{equation*}
{\bs{w_n}}(t):=(w_{n,1}(t),\ldots,w_{n,n}(t)), 
\end{equation*}
which takes values on the simplex
\begin{equation*}
\mathcal{S}_n:=\left\{{\bs{w_n}}=(w_{n,1},\ldots,w_{n,n}):nw_{n,i}\in\{1,\ldots,n\}\text{ for }i\in\{1,\ldots,n\},\sum_{k=1}^nkw_{n,k}=1\right\}.
\end{equation*}Another natural quantity is the empirical cluster size distribution, defined by
\begin{equation}
  p_{n,k}(t):=\frac{\#\{\text{clusters of size $k$ at time $t$}\}}{\#\{\text{clusters at time $t$}\}},\quad 1\leq k \leq n.
  \label{p_def}
\end{equation}
From the description of the dynamics at the beginning of this subsection we can write down the infinitesimal generator of $\bs{w_n}(t)$ by summing over all possible events. To do this we must consider the number of possible ways of coalescing clusters of sizes $l_1,\ldots,l_k$. If the $l_i$ are distinct then we simply multiply together the number of clusters of each size. If some of the $l_i$ are the same then there is a complication as a cluster cannot coalesce with itself. To each configuration $l_1,\ldots,l_k$ we can associate a partition $\pi$ of $\{1,\ldots,k\}$ by $l_u=l_v$ if and only if $u,v\in\pi_i$ for some $i$. In this way the set $\mathcal{P}_k$, of all partitions of $\{1,\ldots,k\}$, enumerates all the ways we could have chosen $k$ cluster sizes with replacement. Hence,
\begin{align}
\mathcal{A}_nf(\bs{w_n})&=\lambda_n n\sum_{k=1}^n\left(f\left(\bs{w_n}-\frac{1}{n}\bs{e_{n,k}}+\frac{k}{n}\bs{e_{n,1}}\right)-f(\bs{w_n})\right)w_{n,k}\nonumber\\
&\qquad +\sum_{k=2}^n\frac{\alpha(k)}{k!n^{k-1}}\sum_{\pi\in\mathcal{P}_k}\sum_{l_1=1}^n\cdots\sum_{l_{\abs{\pi}}=1}^n\gamma_n(f,\pi,\bs{w_n}), \label{infgen}
\end{align}
where 
\begin{equation}
\gamma_n(f,\pi,\bs{w_n}):=\left(f\left(\bs{w_n}-\frac{1}{n}\sum_{i=1}^{\abs{\pi}}\abs{\pi_i}\bs{e_{n,l_i}}+\frac{1}{n}\bs{e_{n,\sum \abs{\pi_i} l_i}}\right)-f(\bs{w_n})\right)\prod_{i=1}^{\abs{\pi}}(nw_{n,l_i})_{\abs{\pi_i}},\label{gamma}
\end{equation}
where $\bs{e_{n,i}}$ is a vector of length $n$ with one in position $i$ and zero everywhere else, $\abs{\pi}$ is the number of blocks in the partition, $\abs{\pi_i}$ is the number of elements in the $i^{th}$ block of $\pi$ and $(y)_z$ is the falling factorial Pochhammer symbol.

Rather than working with $\bs{w_n}$ directly, considerable simplification is possible using the empirical generating functions $G_n:[0,1]\times\mathbb{R}^+\to\mathbb{R}$
\begin{equation}\label{bigG}
  G_n(x,t)=\sum_{k=1}^nx^kw_{n,k}(t),\qquad n\geq 1,
\end{equation}
and $g_n:\![0,1]\times\mathbb{R}^+\to\mathbb{R}$
\begin{equation}\label{littleg}
g_n(x,t):=\sum_{k=1}^nx^kp_{n,k}(t)=\frac{G_n(x,t)}{G_n(1,t)},\qquad n\geq 1\,.
\end{equation}

\subsection{Main Results}
Our first main result shows that processes in the class outlined in \ref{class} are self-averaging with respect to the distribution of cluster sizes as $n\to\infty$.

\begin{theorem}\label{thm1}
Suppose that the coalescence rates $\alpha:\mathbb{N}\to\mathbb{R}^+$ satisfy
\begin{equation}
  \label{alpha}
  \alpha(k)\leq C\exp(\gamma k\ln\ln(k))\,,\qquad \forall k\,,
\end{equation}
for some constants $C>0$ and $\gamma<1$. Define $\lambda=\lim_{n\to\infty}\lambda_n$, and let $G:\![0,1]\times\mathbb{R}^+\to\mathbb{R}$ be the unique solution of the deterministic initial value problem
\begin{equation}
\begin{gathered}
G(x,0)=x\,,\\
  \frac{\partial G}{\partial t}(x,t)=\lambda(x-G(x,t))+\sum_{k=2}^{\infty}\frac{\alpha(k)}{k!}\left(G(x,t)^k-kG(1,t)^{k-1}G(x,t)\right).
\end{gathered}\label{hatG}
\end{equation}

If $\lambda>0$ then in the limit $n\to \infty$ the empirical generating function $G_n(x,t)$ defined in (\ref{bigG}) converges to $G(x,t)$ in $L^2$, uniformly in $x$ and $t$. That is,
\begin{equation*}
  \sup_{x\in[0,1], t\in\mathbb{R}^+}\mathbb{E}\left[(G(x,t)-G_n(x,t))^2\right]\leq\frac{H}{\sqrt{n}}\to0,\quad\textrm{as}\quad n\to\infty,
\end{equation*}
for some constant $H$.
If $\lambda=0$ then the above result holds for $t$ less than any fixed finite time $T$. 

\end{theorem}
\begin{remark}
It should be noted that condition \eqref{alpha} is a sufficient technical condition, but may not be minimal. In addition, it is not known if $n^{-1/2}$ is the optimal rate of convergence.
\end{remark}
Unpacking the limiting generating function $G(x,t)$ allows us to write Smoluchowski-type differential equations for the (rescaled) number of clusters of a given size. 
\begin{corollary}\label{smol}
Under the assumptions of theorem \ref{thm1}, we have $L^2$ convergence $w_{n,j}(t)\to w_j(t)$, obeying
\begin{equation}
\begin{split}
w_j(0)&=\delta_{j,1},\\
  \frac{d}{d t}w_1(t)&=\lambda(1-w_1(t))-w_1(t)\sum_{k=2}^\infty\frac{\alpha(k)}{(k-1)!}\left(\sum_{l=1}^\infty w_l(t)\right)^{k-1},\\
\frac{d}{d t}w_j(t)&=-\lambda w_j(t)-w_j(t)\sum_{k=2}^\infty\frac{\alpha(k)}{(k-1)!}\left(\sum_{l=1}^\infty w_l(t)\right)^{k-1}\!\!+\sum_{k=2}^j\frac{\alpha(k)}{k!}\!\sum_{\substack{l_1,\ldots,l_k\\ l_1+\cdots+l_k=j}}w_{l_1}(t)\cdots w_{l_k}(t).
\end{split}\label{wj}
\end{equation}
for $j\geq2$. Furthermore, the limiting cluster size distribution may be obtained via $p_k(t)=w_{k}(t)/\sum_{l=1}^\infty w_l(t)$.
\end{corollary}

Our second main result deals with how the limiting process behaves in long times when fragmentation events are rare. To that end, we define the stationary cluster size distribution $$p_k:=\lim_{t\to\infty}\lim_{n\to\infty}p_{n,k}(t)\,,$$
where $p_{n,k}(t)$ is defined in (\ref{p_def}) and convergence in the large $n$ limit should be understood in terms of Theorem \ref{thm1}.

\begin{theorem}\label{thm2}
If $\alpha$ satisfies  \eqref{alpha} and $m$ is the smallest integer such that $\alpha(m)>0$, then the stationary cluster size distribution exists and is unique. In particular, the limit of this stationary distribution as $\lambda\to0$ exists and obeys
\begin{equation*}
  \lim_{\lambda\searrow0}p_k=\begin{cases}\frac{1}{k}\left(\frac{m-1}{m}\right)^k\left(\frac{1}{m}\right)^{\frac{k-1}{m-1}}\left(\bfrac{m\left(\frac{k-1}{m-1}\right)}{\frac{k-1}{m-1}}\right)&\mbox{if } m-1\text{ divides }k-1\\
0 &\mbox{otherwise.}\end{cases}
\label{pdist}
\end{equation*}
In particular $\lim_{\lambda\searrow0}p_k\sim k^{-3/2}$, if $m-1$ divides $k-1$, where $f\sim g$ means $f(k)/g(k)\to c$ for some positive constant $c$, as $k\to\infty$, regardless of $\alpha$.

\end{theorem}
\begin{figure}
\begin{center}
\includegraphics[width=0.35\textwidth, trim=0 -40 -20 0]{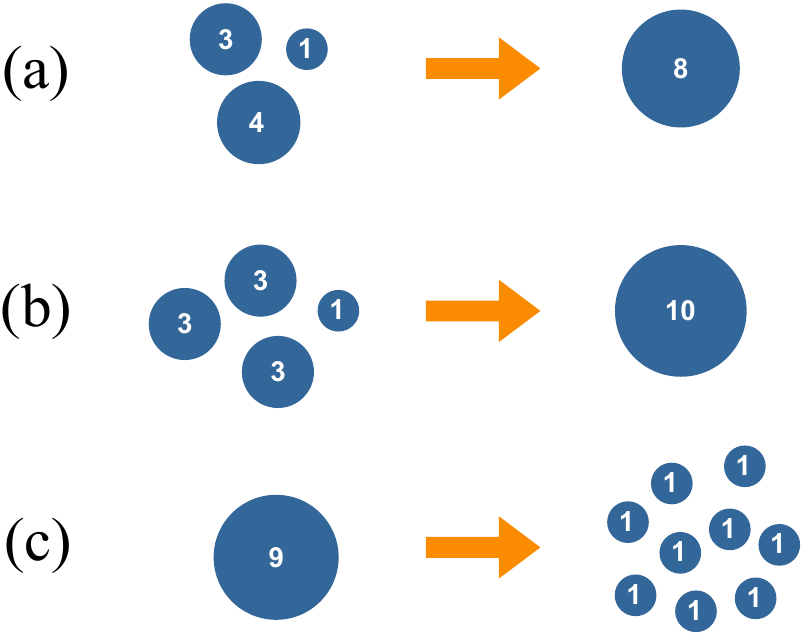}\includegraphics[width=0.4\textwidth, trim=0 10 0 0]{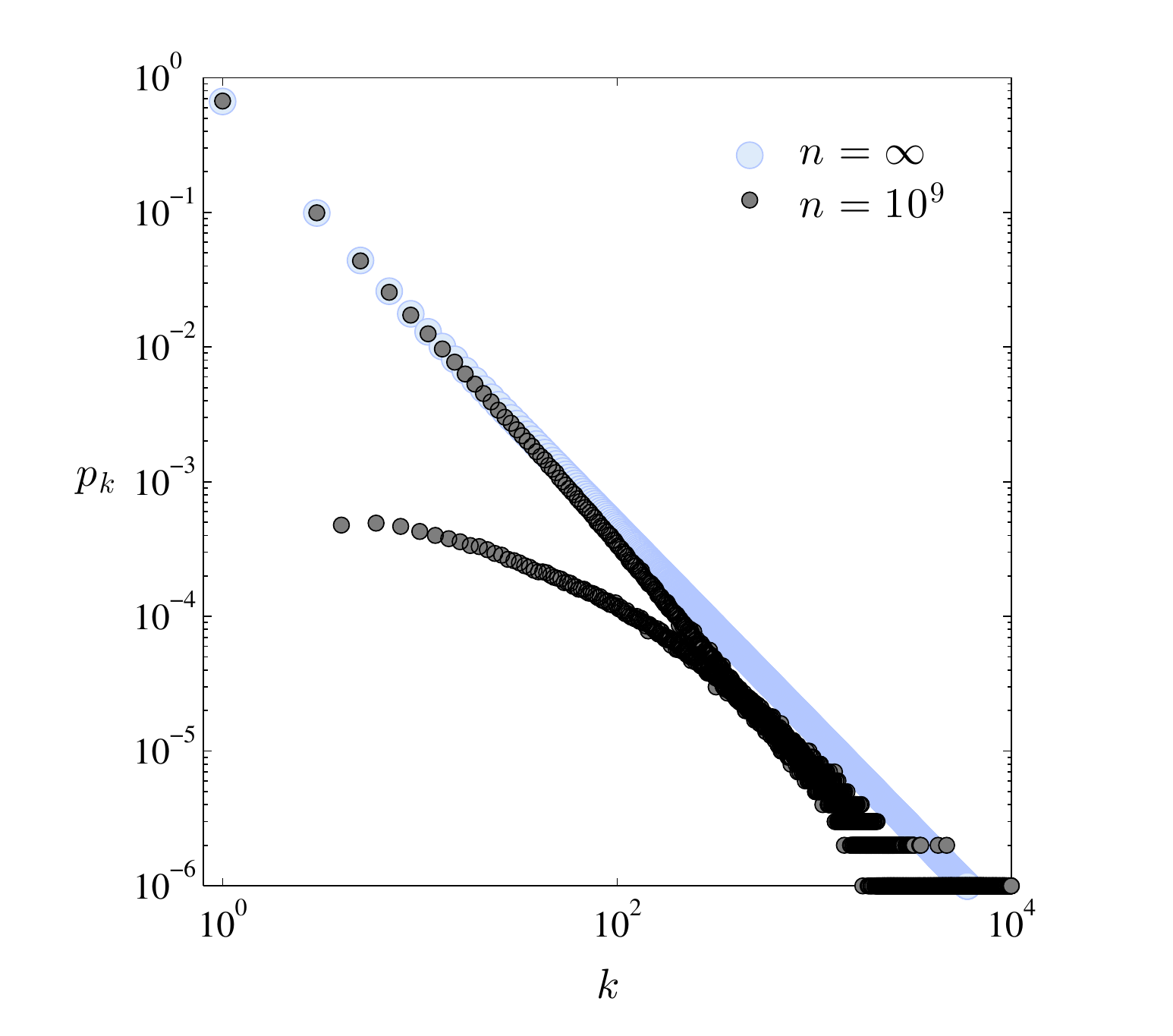}
\end{center}
\caption{Example of a fragmentation-coalescence process with fragmentation rate $\lambda=10^{-7}$ and coalescence rates $\alpha(3)=1$, $\alpha(4)=2$ and $\alpha(k)=0$ otherwise. Left: illustration of possible events (a) coalescence of three clusters, (b) coalescence of four clusters, (c) fragmentation of a cluster. Right: comparison between a single simulation run with $n=10^9$ particles, stopped at time $T=10^6$ (black) and the $n\to\infty$ limit given by Theorem \ref{thm2} (pale blue). The higher points correspond to the odd sized blocks, the lower ones correspond to even sized blocks which are asymptotically vanishing in the limit $n\to\infty$, $\lambda\to0$.}
\label{fig1}
\end{figure}
The power-law tail of cluster size distribution is reminiscent of critical behaviour occurring at the gellation time in some pure coalescent processes \cite{Aldous1999}, although here it is not a transient phenomenon but rather is approached in the long time limit. This is the essential characteristic of self-organised criticality, however, there is an unusual complication here that the result depends on the strict order of the $n\to\infty$ and $\lambda\to0$ limits; taking $\lambda_n\to0$ we recover a pure coalescent process that does not posses this characteristic power-law. 

While it is quite typical for critical processes to exhibit universality in the scaling exponents, our result says something stronger that in fact the limiting cluster size distribution is almost completely independent of the coagulation rates, depending only on the identity of the first non-zero rate. A surprising consequence is that if, for example, the model coalesces clusters in groups of three and four (but not pairs) then in the large $n$ and small $\lambda$ limit we will see no clusters of even size whatsoever in the stationary distribution. Figure~\ref{fig1} shows an example of this phenomenon for the model with rates $\alpha(k)=\delta_{k,3}+2\delta_{k,4}\,.$ The model has the apparently paradoxical feature that clusters of even size are vanishingly rare, despite the fact that some two-thirds of clusters are singletons, and $\alpha(4)>\alpha(3)$. 

We give here a brief heuristic as to why this is the case. The coalescence rates are scaled so that when the number of clusters is of order $n$ the total coalescence rate of events involving $k$ blocks is of the same order for all $k$. However, Theorem \ref{thm2} states as $\lambda\to0$ that the stationary cluster size distribution follows a power law with exponent 3/2, which suggests that the average cluster has size of order $n^{1/2}$. Hence, the number of clusters falls below order $n$. This causes coalescence events of the smallest number of clusters, $m$, to occur much more frequently than all others and so eventually we only see clusters that could have arised from coalescing $m$ clusters. These are clusters of size $k$ where $m-1$ divides $k-1$. In the above example where $m=3$, this results in the distribution only giving mass to odd-sized blocks.

\subsection{Examples of related work}
Without fragmentation one is left with a pure coalescent process \cite{Aldous1999}. The prototypical example of this class has pairs of clusters coalescing at unit rate. In this case, the $w_{n,k}(t)$ converge in probability to solutions to the Smoluchowski coagulation equations \cite{Smoluchowski1916} with unit rate kernel:
\begin{equation*}
\frac{d}{dt}w_k(t)=\frac{1}{2}\sum_{j=1}^{k-1}w_j(t)w_{k-j}(t)-w_k(t)\sum_{j=1}^{\infty}w_j(t),\quad\quad\text{$k\geq1$}.
\end{equation*}
These equations can be solved to show
\begin{equation*}
w_k(t)=\left(1+\frac{t}{2}\right)^{-2}\left(\frac{t}{2+t}\right)^{k-1},\qquad k\geq 1,\quad t\geq 0,
\end{equation*}
and so the proportion of clusters of size $k$ decays exponentially in $k$. This behaviour is typical of pure coalescence processes.

A simple modification of these dynamics is to sized-biased fragmentation-coalescence processes, a particular example of which (again with only pairs coalescing) has been studied in the context of terrorist networks \cite{Bohorquez2009}. A thorough analysis of this model was undertaken by R\'ath and T\'oth in \cite{RathToth2009}. For a constant fragmentation rate $\lambda$ it was found that the $w_{n,k}(t)$,  converge in probability as $n\to\infty$ to $w_k(t)$ solving a set of differential equations similar to those of Smoluchowski:
\begin{gather*}
  \frac{d}{dt}w_k(t)=\frac{1}{2}\sum_{j=1}^{k-1}j(k-j)w_{j}(t)w_{k-j}(t)-kw_{k}(t)-\lambda kw_{k}(t),\quad \text{$k\geq2$},\\
\frac{d}{dt}w_1(t)=-w_{1}(t)+\lambda\sum_{j=2}^{\infty}j^2w_{j}(t).
\end{gather*}
And so it was shown in the long time stationary limit, as $t\to\infty$, that 
\begin{equation*}
w_k\sim \exp\left(k\log\left(1-\frac{\lambda^2}{(1+\lambda)^2}\right)\right)k^{-5/2}.
\end{equation*}
As in our model a power-law tail emerges in the limit $\lambda\to0$, this time with exponent $5/2$. Interestingly in this case the critical behaviour survives the joint limit $\lambda_n\to0$ as $n\to\infty$; the main effort of \cite{RathToth2009} was focused on this scaling regime.

The critical behaviour of the stationary cluster size distribution, found in Theorem \ref{thm2}, has also been seen in a different form of fragmentation-coalescence process studied by Bressaud and Fournier \cite{BressaudFournier2014}. This was a mean-field approximation to a one-dimensional forest-fire model on $\mathbb{Z}$. Edges (each of weight one) are glued together if the site between them is occupied, which happens at rate 1, and clusters of $k$ edges have all sites removed at rate $(k-1)/n$. A mean-field approximation to this model, ignoring correlations, follows
\begin{gather*}
\frac{d}{dt} w^{(n)}_1(t)=-2w^{(n)}_1(t)+\frac{1}{n}\sum_{k\geq2}k(k-1)w^{(n)}_k(t),\\
\frac{d}{dt} w^{(n)}_k(t)=-(2+(k-1)/n)w^{(n)}_k(t)+\left(\sum_{l\geq1}w^{(n)}_l(t)\right)^{-1}\sum_{i=1}^{k-1}w^{(n)}_i(t)w^{(n)}_{k-i}(t), \quad k\geq2.
\end{gather*}
Then, they found that the stationary cluster size distribution converges weakly to $(p_k)_{k\geq1}$ as $n\to\infty$ (`fragmentation' rate converges to zero) where 
\begin{equation*}
p_k=\frac{2}{4^kk}\binom{2k-2}{k-1}
\end{equation*}
which matches the distribution in Theorem \ref{thm2} in the case where $m=2$.

Finally we mention an important separate class of fragmentation-coalescence processes. The models discussed so far all concern large but finite particle systems, but it is possible to define (non-size-biased) fragmentation and coalescence processes on the partitions of $\mathbb{N}$, as introduced in \cite{BerestyckiJ2004}. The class of processes we work with here can be seen as a natural finite counterpart to these models. 

\section{Thermodynamic limit}
We will focus our proof on the case of fixed fragmentation rate $\lambda_n\equiv\lambda\in(0,\infty)$; the extension to the joint limit differs only in a few places.  Besides standard generating function technology, our methods rely on Gronwall's Inequality \cite{Verhulst1990} which we reproduce below for convenience.

\begin{lemma}[Gronwall's Inequality]\label{gron}
Suppose $f:[0,\infty)\to\mathbb{R}$ is differentiable and satisfies the following differential inequality
\begin{equation*}
\frac{d}{dt}f(t)\leq af(t)+b,\qquad t> 0,
\end{equation*}
where $a,b\in\mathbb{R}$. Then, we have that for all $t\geq0$
\begin{equation*}
f(t)\leq \begin{cases} f(0)e^{at}+\frac{b}{a}\left(e^{at}-1\right) &\mbox{if } a\neq0\\ f(0)+bt &\mbox{if } a=0.\end{cases}
\end{equation*}
\end{lemma}
The idea of the proof of Theorem \ref{thm1} is to bound the derivative of the expected squared difference between $G(x,t)$ and $G_n(x,t)$ in such a way so as we may apply lemma \ref{gron}, where, here, $a$ will be negative and $b=b_n$ will decay to zero as $n\to\infty$. To achieve this, we will need to prove that the derivative exists and certain error quantities (specified in the next section) converge to zero as $n$ tends to infinity. 

\subsection{Mean-Field Calculation}
We first undertake a mean-field calculation to determine a viable candidate for $\lim_{n\to\infty}G_n$ under self-averaging.
\begin{lemma}\label{lemma:A}
For $f_0(x,\bs{w_n}):=\sum_{k=1}^nx^kw_{n,k}$, we have
\begin{equation*}
\mathcal{A}_nf_0(x,\bs{w_n})
=\lambda(x-f_0(x,\bs{w_n}))+\sum_{k=2}^n\frac{\alpha(k)}{k!}(f_0(x,\bs{w_n})^k-kf_0(1,\bs{w_n})^{k-1}f_0(x,\bs{w_n}))+\beta_n(x,\bs{w_n}),
\end{equation*}
where
\begin{equation*}
\sup_{\bs{w_n}\in\mathcal{S}_n}\abs{\beta_n(x,\bs{w_n})}\leq \frac{A}{n},
\end{equation*}
where $A$ is a constant independent of $n$ and $x$.
\end{lemma}

\begin{proof}
We use the infinitesimal generator from \eqref{infgen} in the special case where $f=f_0$ using a slight abuse of notation. Hence

\begin{align*}
\mathcal{A}_nf_0(x,\bs{w_n})&=-\lambda\sum_{k=1}^nx^kw_{n,k}+\lambda\sum_{k=1}^{n}kxw_{n,k}\\
&\qquad+\sum_{k=2}^n\frac{\alpha(k)}{k!n^k}\sum_{\pi\in\mathcal{P}_k}\sum_{l_1=1}^n\cdots\sum_{l_{\abs{\pi}}=1}^n\left(\prod_{i=1}^{\abs{\pi}}x^{\abs{\pi_i} l_i}-\sum_{i=1}^{\abs{\pi}}\abs{\pi_i} x^{l_i}\right)\prod_{i=1}^{\abs{\pi}}(nw_{n,l_i})_{\abs{\pi_i}}.
\end{align*}
As $n$ grows large, we claim that the dominant contribution from the final product inside the third term on the right-hand side is simply $n^kw_{n,l_1}\cdots w_{n,l_k}$. If $\pi=\{1,\ldots,k\}$, then this is the only term, otherwise there is a subdominant correction resulting from the fact that clusters cannot coagulate with themselves. Labelling this correction as $\beta_n(x,\bs{w_n})$, which we will later bound, we obtain
\begin{align}\label{DCT}
 \mathcal{A}_nf_0(x,\bs{w_n})&=-\lambda\sum_{k=1}^nx^kw_{n,k}+\lambda\sum_{k=1}^{n}kxw_{n,k}\nonumber\\
&\quad\quad+\sum_{k=2}^n\frac{\alpha(k)}{k!}\sum_{l_1=1}^n\cdots\sum_{l_k=1}^n\left(\prod_{i=1}^kx^{l_i}-\sum_{i=1}^kx^{l_i}\right)w_{n,l_1}\cdots w_{n,l_k}\nonumber\\
&\quad\quad+\beta_n(x,\bs{w_n})\nonumber\\
&=\lambda(x-f_0(x,\bs{w_n}))+\sum_{k=2}^n\frac{\alpha(k)}{k!}\left(f_0(x,\bs{w_n})^k-kf_0(1,\bs{w_n})^{k-1}f_0(x,\bs{w_n})\right)+\beta_n(x,\bs{w_n}).
\end{align}

It remains to show that $\beta_n(x,\bs{w_n})\leq A/n$ for some constant $A>0$. We know it is the remainder after subtracting the dominant term and so is
\begin{equation*}
\beta_n(x,\bs{w_n})=\sum_{k=2}^n\frac{\alpha(k)}{k!}\sum_{\pi\in\mathcal{P}_k}\sum_{l_1=1}^n\cdots\sum_{l_{\abs{\pi}}=1}^n\left(\prod_{i=1}^{|\pi|}x^{\abs{\pi_i}l_i}-\sum_{i=1}^{\abs{\pi}}\abs{\pi_i}x^{l_i}\right)\left(\prod_{i=1}^{\abs{\pi}}\frac{(nw_{n,l_i})_{\abs{\pi_i}}}{n^{\abs{\pi_i}}}-\prod_{i=1}^{\abs{\pi}}w_{n,l_i}^{\abs{\pi_i}}\right).
\end{equation*}
Each multinomial in the second bracket, when expanded, has at most $2^k$ terms with varying powers of $1/n$ larger than or equal to one. Also, we can see that all the terms are divisible by $w_{n,l_1}\cdots w_{n,l_{\abs{\pi}}}$ and have a coefficient that is less than or equal to $(k-1)^j$, where $j$ is the power of $1/n$ in that term. Hence, 
\begin{equation*}
  |\beta_n(x,\bs{w_n})|\leq\sum_{k=2}^n\frac{\alpha(k)}{(k-1)!}\sum_{\pi\in\mathcal{P}_k}\sum_{l_1=1}^n\cdots\sum_{l_{\abs{\pi}}=1}^n \frac{2^k (k-1)}{n}w_{n,l_1}\cdots w_{n,l_{\abs{\pi}}}.
\end{equation*}
Summing over $l_1,\ldots,l_{\abs{\pi}}$ we have that,
\begin{equation*}
|\beta_n(x,\bs{w_n})|\leq\frac{1}{n}\sum_{k=2}^n\frac{\alpha(k)2^k}{(k-2)!}\sum_{\pi\in\mathcal{P}_k}f_0(1,\bs{w_n})^{\abs{\pi}}.
\end{equation*}
Then noting that $f_0(1,\bs{w_n})^{\abs{\pi}}\leq1$ for all $\pi$, we have
\begin{equation*}
  |\beta_n(x,\bs{w_n})|\leq\frac{1}{n}\sum_{k=2}^n\frac{\alpha(k)2^k}{(k-2)!}B_k,
\end{equation*}
where $B_k$ denotes the $k^{th}$ Bell number and is the size of the set $\mathcal{P}_k$. We bound $B_k$ using a recent result from \cite{BerendTassa2010}, obtaining
\begin{equation*}
  |\beta_n(x,\bs{w_n})|\leq\frac{1}{n}\left|\sum_{k=2}^n\frac{\alpha(k)}{(k-2)!}\left(\frac{1.584k}{\ln(k+1)}\right)^k\right|.
\end{equation*}
If we look at the asymptotic behaviour of the summands for large $k$, we see that 
\begin{equation*}
  \frac{1}{(k-2)!}\left(\frac{1.584k}{\ln(k+1)}\right)^k\sim \exp\left(-k\ln\ln(k)+O(k)\right),
\end{equation*}
where $\sim$ is defined in Theorem \ref{thm2}. Thus, by assumption \eqref{alpha} on $\alpha$, the series converges as required.
\end{proof}
\vspace{0.2cm}

\subsection{Self-averaging}
With the mean-field behaviour determined, we proceed to the proof of $L^2$ convergence for sample paths. The following two lemmas will help us determine the behaviour of $G_n^2$.

\begin{lemma}\label{lemma:C}

The function $f_0(x,\bs{w_n})=\sum_{k=1}^nx^kw_{n,k}$ satisfies
\begin{equation*}
\mathcal{A}_nf_0^2(x,\bs{w_n})=2f_0(x,\bs{w_n})\mathcal{A}_nf_0(x,\bs{w_n})+\frac{\lambda x^2}{n}f_2(1,\bs{w_n})+h(x,\bs{w_n}),
\end{equation*}
where
\begin{equation*}
f_2(x,\bs{w_n}):=\sum_{k=1}^nk^2x^kw_{n,k},\quad\text{and}\quad
\sup_{\bs{w_n}\in\mathcal{S}_n}h(x,\bs{w_n})\leq\frac{E}{n},
\end{equation*}
such that $E$ is a constant independent of $n$ and $x$.
\end{lemma}
\begin{proof}
We plug the function $f_0^2$ into the infinitesimal generator \eqref{infgen}
\begin{align*}
\mathcal{A}_nf_0^2(x,\bs{w_n})&=\lambda_n n\sum_{k=1}^n\left(\left(f_0(x,\bs{w_n})+\frac{kx-x^k}{n}\right)^2-f_0^2(x,t)\right)w_{n,k}\nonumber\\
&\qquad +\sum_{k=2}^n\frac{\alpha(k)}{k!n^{k-1}}\sum_{\pi\in\mathcal{P}_k}\sum_{l_1=1}^n\cdots\sum_{l_{\abs{\pi}}=1}^n\gamma_n(f_0^2,\pi,\bs{w_n}).
\end{align*}
where $\gamma$ is defined in \eqref{gamma}. By expanding out the brackets we see that

\begin{equation*}
\mathcal{A}_nf_0^2(x,\bs{w_n})=2f_0(x,\bs{w_n})\mathcal{A}_nf_0(x,\bs{w_n})+\frac{\lambda x^2}{n}f_2(1,\bs{w_n})+h(x,\bs{w_n}),
\end{equation*}
where
\begin{align*}
  h(x,\bs{w_n})&=\frac{\lambda}{n}\sum_{k=1}^n\left(x^{2k}-2kx^{k+1}\right)w_{n,k}\\
&\qquad +\frac{1}{n}\sum_{k=2}^n\frac{\alpha(k)}{k!}\sum_{\pi\in\mathcal{P}_k}\sum_{l_1=1}^n\cdots\sum_{l_{\abs{\pi}}=1}^n\left(\prod_{i=1}^{\abs{\pi}}x^{\abs{\pi_i} l_i}-\sum_{i=1}^{\abs{\pi}}\abs{\pi_i} x^{l_i}\right)^2\prod_{i=1}^{\abs{\pi}}\frac{(nw_{n,l_i})_{\abs{\pi_i}}}{n^{\abs{\pi_i}}}\\
&\leq -\frac{2\lambda}{n}x^2+\frac{\lambda}{n}f_0(x^2,\bs{w_n})+\frac{1}{n}\sum_{k=2}^n\frac{\alpha(k)}{k!}f_0(x^2,\bs{w_n})^k\\
&\qquad-\frac{2}{n}\sum_{k=2}^n\frac{\alpha(k)}{k!}kf_0(x^2,\bs{w_n})f_0(x,\bs{w_n})^{k-1}+\frac{1}{n}\sum_{k=2}^n\frac{\alpha(k)}{k!}kf_0(x^2,\bs{w_n})f_0(1,\bs{w_n})^{k-1}\\
&\qquad+\frac{1}{n}\sum_{k=2}^n\frac{\alpha(k)}{(k-2)!}f_0(x,\bs{w_n})^2f_0(1,\bs{w_n})^{k-2}+\frac{1}{n^2}\sum_{k=2}^n\frac{\alpha(k)k}{(k-2)!}\left(\frac{1.584k}{\ln(k+1)}\right)^k.
\end{align*}
where the bound was attained by using Lemma \ref{lemma:B} and the same techniques as in the proof of lemma \ref{lemma:A}. Now, using the bound $f_0(x^j,\bs{w_n})^l\leq1$ for all $j,l\geq1$, and \eqref{B}, we obtain
\begin{align*}
  h(x,\bs{w_n})&\leq\frac{1}{n}\left(\lambda +\sum_{k=2}^{n}\alpha(k)\left(\frac{1}{k!}+\frac{3}{(k-1)!}+\frac{2}{(k-2)!}\right)+\frac{B}{n}\right)\leq\frac{E}{n},
\end{align*}
where $E$ is independent of $n$ and $x$, as required.
\end{proof}
Now we need to bound the expected value of the function $f_2(1,\bs{w_n}(t))$.
\begin{lemma}\label{lemma:B}
Define the function $C_n:[0,\infty)\to\mathbb{R}$ as follows
\begin{equation*}
C_n(t)=\mathbb{E}\left[f_2(1,\bs{w_n}(t))\right],
\end{equation*}
for $n\geq 1$ and $t\geq 0$.
Then we have that
\begin{equation*}
\sup_{n\in\mathbb{N}}C_n(t)\leq D,
\end{equation*}
where $D$ is a constant independent of $t$.
\end{lemma}
\vspace{0.2cm}
\begin{proof}
We will use Gronwall's Inequality. We look at the derivative of $C_n$, using Dynkin's formula we see that
\begin{equation*}
\frac{d}{dt}C_n(t)=\mathbb{E}\left[\mathcal{A}_nf_2(1,\bs{w_n}(t))\right],
\end{equation*}
and hence, using the infinitesimal generator \eqref{infgen},
\begin{align*}
\frac{d}{dt}C_n(t)&=\mathbb{E}\Bigg[-\lambda\sum_{k=1}^n(k^2-k)w_{n,k}(t)\\
& \qquad+\sum_{k=2}^n\frac{\alpha(k)}{k!}\sum_{\pi\in\mathcal{P}_k}\sum_{l_1=1}^n\cdots\sum_{l_{\abs{\pi}}=1}^n\left[\left(\sum_{i=1}^{\abs{\pi}}\abs{\pi_i} l_i\right)^2-\sum_{i=1}^{\abs{\pi}}(\abs{\pi_i} l_i)^2\right]\prod_{i=1}^{\abs{\pi}}\frac{(nw_{n,l_i}(t))_{\abs{\pi_i}}}{n^{\abs{\pi_i}}}\Bigg].
\end{align*}
which, using a similar argument to that used when bounding $\beta_n(x,t)$, gives that

\begin{align*}
\frac{d}{dt}C_n(t)&\leq\lambda-\lambda C_n(t)\\ 
&\qquad +2\mathbb{E}\left[\sum_{k=2}^n\frac{\alpha(k)}{k!}\sum_{l_1=1}^n\cdots\sum_{l_k=1}^n\left(\sum_{i<j}l_il_j\right)w_{n,l_1}(t)\cdots w_{n,l_k}(t)\right]\\
&\qquad +2\mathbb{E}\left[\sum_{k=2}^n \frac{\alpha(k)}{k!}\frac{2^k(k-1)}{n}\sum_{\pi\in\mathcal{P}_k}\sum_{l_1=1}^n\cdots\sum_{l_{\abs{\pi}}=1}^n\left(\sum_{i<j}\abs{\pi_i}\abs{\pi_j} l_i l_j\right)w_{n,l_1}(t)\cdots w_{n,l_{\abs{\pi}}}(t)\right]\\
&\leq \lambda -\lambda C_n(t)+2\mathbb{E}\left[\sum_{k=2}^n\frac{\alpha(k)}{k!}{\binom{k}{2}}G_n(1,t)^{k-2}+\frac{1}{n}\sum_{k=2}^n \frac{\alpha(k)k}{k!}\sum_{\pi\in\mathcal{P}_k}\binom{k}{2}G_n(1,t)^{\abs{\pi}-2}\right]\\
&\leq \lambda -\lambda C_n(t)+\sum_{k=2}^n\frac{\alpha(k)}{(k-2)!}+\frac{B}{n},
\end{align*}
where 
\begin{equation}\label{B}
B=\sum_{k=2}^\infty \frac{\alpha(k)k}{(k-2)!}\left(\frac{1.584k}{\ln(k+1)}\right)^k.
\end{equation}
Thus,

\begin{equation*}
\frac{d}{dt}C_n(t)\leq D_1-\lambda C_n(t)
\end{equation*}
where $D_1$ is some constant not dependent on $n$ or $t$, so applying Gronwall's inequality gives the result as required.

\end{proof}
With the help of the above three lemmas we are now ready to prove Theorem \ref{thm1}.
\begin{proof}[Proof of Theorem 1]\renewcommand{\qedsymbol}{}

Recall from \eqref{hatG} that we write $G(x,t)$ for the solution of the differential equation 
\begin{gather*}
  \frac{\partial G}{\partial t}(x,t)=\lambda(x-G(x,t))+\sum_{k=2}^{\infty}\frac{\alpha(k)}{k!}\left(G(x,t)^k-kG(1,t)^{k-1}G(x,t)\right)\\
G(x,0)=x.
\end{gather*}
The existence and uniqueness of $G(x,t)$ is straightforward to establish via the Picard-Lindeh\"of Theorem, as the right hand side is uniformly Lipschitz in $G$, $x$ and $t$. Specifically, starting with the case $x=1$, we have
\begin{equation}
\frac{\partial G(1,t)}{\partial t}=F_1(G(1,t))\,,\label{F1}
\end{equation}
where
\begin{equation}
F_1(g)=\lambda(1-g)+\sum_{k\geq 2}\frac{\alpha(k)}{k!}(1-k)g^k\,.
\end{equation}
The derivative of $F_1$ is uniformly bounded on $g\in[0,1]$ since
\begin{equation}
F'_1(g)=-\lambda-\sum_{k\geq 2}\frac{\alpha(k)}{(k-2)!}g^{k-1}\,,
\end{equation}
and thus
\begin{equation}
\sup_{g\in[0,1]}\left|F_1'(g)\right|=\lambda+\sum_{k\geq 2}\frac{\alpha(k)}{(k-2)!}=L_1<\infty.
\end{equation}
Said another way, $F_1$ is uniformly Lipschitz with constant $L_1$, and thus \eqref{F1} has a unique solution. A very similar argument using furthermore the well-posedness of \eqref{F1} applies to the case of general $x\in(0,1]$, where we have
\begin{equation}
\frac{\partial G(x,t)}{\partial t}=F_x(G(x,t),t)\,,
\end{equation}
for
\begin{equation}
F_x(g,t)=\lambda(x-g)+\sum_{k\geq 2}\frac{\alpha(k)}{k!}\big(g^k-kG(1,t)^{k-1}g\big)\,.
\label{Fx}
\end{equation}
The derivative is again bounded:
\begin{equation}
F'_x(g,t)=-\lambda+\sum_{k\geq 2}\frac{\alpha(k)}{(k-1)!}\big(g^{k-1}-G(1,t)^{k-1}\big)\,,
\end{equation}
and since $G(1,t)\in[0,1]$ we obtain 
\begin{equation}
\sup_{g\in[0,1],t\geq0}\left|F_x'(g,t)\right|=\lambda+2\sum_{k\geq 2}\frac{\alpha(k)}{(k-1)!}=L<\infty.
\end{equation}
As this constant is independent of $x,g$ and $t$, global existence and uniqueness follows. 

Moving on now to compare to the finite system generating function $G_n(x,t)$, we start by noting that in the trivial case $x=0$, we have 
\begin{equation*}
G_n(0,t)=G(0,t)=0, \quad \forall t\geq 0\,.
\end{equation*}
Now, for strictly positive $x$ we define the function $Y_{n,x}(t):[0,\infty)\to\mathbb{R}$ as follows:
\begin{equation*}
Y_{n,x}(t)=\mathbb{E}\big[(G(x,t)-G_n(x,t))^2\big], \quad t\geq0.
\end{equation*}
We will use Gronwall's Inequality, so we look at the derivative of $Y_{n,x}(t)$ with respect to $t$, using Dynkin's formula (bearing in mind that pre-limit in $n$, the system is finite) we see that
\begin{align}
\frac{dY_{n,x}}{dt}(t)&=\mathbb{E}\left[\mathcal{A}_n(G(x,t)-G_n(x,t))^2\right]\nonumber\\
&=2G(x,t)\frac{\partial}{\partial t}G(x,t)+\mathbb{E}\left[\mathcal{A}_nG_n^2(x,t)\right]
-2G(x,t)\mathbb{E}\left[\mathcal{A}_nG_n(x,t)\right]\nonumber\\&\qquad-2\mathbb{E}[G_n(x,t)]\frac{\partial}{\partial t}G(x,t).\label{yn}
\end{align}
Noting that $G_n(x,t)=f_0(x,\bs{w_n}(t))$, we may use lemma \ref{lemma:A}, lemma \ref{lemma:C}, and the definition of $G$ in \eqref{hatG} to see that
\begin{align}
\frac{d}{dt}Y_{n,x}(t)&=2G(x,t)\bigg(\lambda(x-G(x,t))+\sum_{k=2}^{\infty}\frac{\alpha(k)}{k!}\left(G(x,t)^k-kG(1,t)^{k-1}G(x,t)\right)\bigg)\nonumber\\
&\quad\quad+2\mathbb{E}\bigg[G_n(x,t)\bigg(\lambda(x-G_n(x,t))+\beta_n(x,\bs{w_n}(t))\nonumber\\
&\qquad\qquad+\sum_{k=2}^n\frac{\alpha(k)}{k!}(G_n(x,t)^k-kG_n(1,t)^{k-1}G_n(x,t))\bigg)\bigg]\nonumber\\
&\quad\quad-2G(x,t)\mathbb{E}\bigg[\lambda(x-G_n(x,t))+\beta_n(x,\bs{w_n}(t))\nonumber\\
&\quad\quad\quad\quad+\sum_{k=2}^n\frac{\alpha(k)}{k!}(G_n(x,t)^k-kG_n(1,t)^{k-1}G_n(x,t))\bigg]\nonumber\\
&\quad\quad-2\mathbb{E}[G_n(x,t)]\bigg(\lambda(x-G(x,t))+\sum_{k=2}^{\infty}\frac{\alpha(k)}{k!}\left(G(x,t)^k-kG(1,t)^{k-1}G(x,t)\right)\bigg)\nonumber\\
&\quad\quad+\frac{\lambda x^2}{n}C_n(t)+\mathbb{E}\left[h(x,\bs{w_n}(t))\right]\nonumber\\
&=-2\lambda Y_{n,x}(t)+2\mathbb{E}[(G_n(x,t)-G(x,t))\beta_n(x,\bs{w_n}(t))]+\frac{\lambda x^2}{n}C_n(t)+\mathbb{E}[h(x,\bs{w_n}(t))]\nonumber\\
&\quad\quad +2\sum_{k=2}^n\frac{\alpha(k)}{k!}\mathbb{E}\bigg[(G(x,t)-G_n(x,t))\Big(G(x,t)^k-G_n(x,t)^k)\nonumber\\
&\quad\quad\quad\quad-k(G(1,t)^{k-1}G(x,t)-G_n(1,t)^{k-1}G_n(x,t))\Big)\bigg]\nonumber\\
&\quad\quad+2\sum_{k=n+1}^{\infty}\frac{\alpha(k)}{k!}\mathbb{E}\left[(G(x,t)-G_n(x,t))(G(x,t)^k-kG(1,t)^{k-1}G(x,t)\right]
.\label{Y}
\end{align}
\end{proof}
Note that the right hand side of (\ref{Y}) depends on $G_n(1,t)$. For the boundary case $x=1$, however, we have a closed expression in $G_n(1,t)$. The plan is thus to first show the theorem holds for $x=1$ and use this to complete the proof for general $x$. Define $X_n(t):=Y_{n,1}(t)$.
\begin{lemma}\label{lemma:X}
We have
\begin{equation}
\sup_{t\geq0}X_n(t)\leq\frac{H_n}{2\lambda},
\end{equation}
where
\begin{equation*}
H_n:=\frac{1}{n}\left(2A+D\lambda+E+\sum_{k=2}^{\infty}\frac{2\alpha(k)}{(k-2)!}\right),
\end{equation*}
In particular, $H_n\to0$ as $n\to\infty$ because of assumption \eqref{alpha}.
\end{lemma}
\vspace{0.2cm}
\begin{proof}
Substituting $x=1$ into \eqref{Y} we see that
\begin{align*}
\frac{dX_n}{dt}(t)&=-2\lambda X_n(t)+\frac{\lambda}{n}C_n(t)+\mathbb{E}[h(1,\bs{w_n}(t))]+2\mathbb{E}[(G_n(1,t)-G(1,t))\beta_n(1,\bs{w_n}(t))]\nonumber\\
&\quad\quad+2\sum_{k=2}^n\frac{\alpha(k)}{k!}\mathbb{E}\left[(G(1,t)-G_n(1,t))(G(1,t)^k-G_n(1,t)^k)\right](1-k)\nonumber\\
&\quad\quad+2\sum_{k=n+1}^{\infty}\frac{\alpha(k)}{k!}\mathbb{E}[(G(1,t)-G_n(1,t))G(1,t)^{k}(1-k).
\end{align*}
Now,  $(G(1,t)-G_n(1,t))(G(1,t)^k-G_n(1,t)^k)\geq0$, which means the second line above is negative. Hence, together with lemma \ref{lemma:C} and lemma \ref{lemma:B}, we have
\begin{equation}
\frac{dX_n}{dt}(t)\leq -2\lambda X_n(t) + \frac{1}{n}(D\lambda+E+2A)+\sum_{k=n+1}^\infty\frac{2\alpha(k)}{(k-1)!}\leq -2\lambda X_n(t)+H_n,\label{Xchange}
\end{equation}
so that applying Gronwall's inequality gives the result as required.
\end{proof}
We can now use this bound to complete the proof of Theorem \ref{thm1}.
\begin{proof}[Proof of Theorem 1 (cont.)]
Continuing from \eqref{Y}, we have that
\begin{align}
\frac{d}{dt}Y_{n,x}(t)&\leq-2\lambda Y_{n,x}(t) +2\sum_{k=2}^\infty S_k+H_n,\label{Y2}\end{align}
where 
\begin{align*}
S_k&=\frac{\alpha(k)}{k!}\mathbb{E}\bigg[G(x,t)^{k-1}\left(G(x,t)^2-G(x,t)G_n(x,t)\right)+G_n(x,t)^{k-1}\left(G_n(x,t)^2-G(x,t)G_n(x,t)\right)\nonumber\\
&\quad\quad-kG(1,t)^{k-1}\left(G(x,t)^2-G(x,t)G_n(x,t)\right)-kG_n(1,t)^{k-1}\left(G_n(x,t)^2-G(x,t)G_n(x,t)\right)\bigg].
\end{align*}
We will apply lemma \ref{lemma:X} to bound the sum of the $S_k$. First, it is necessary to bound the sum in terms of $\mathbb{E}[|G(1,t)-G_n(1,t)|]$, and remove any terms involving just $G(x,t)$ and $G_n(x,t)$. To do this we create terms that contain the positive term $(G(x,t)-G_n(x,t))^2$, so that if we pre-multiply them by something negative we can discard it for an upper bound. For example, looking at the first term above we can do the following by adding a zero
\begin{align*}
G(x,t)^{k-1}\mathbb{E}\left[G(x,t)^2-G(x,t)G_n(x,t)\right]&=G(x,t)^{k-1}\mathbb{E}\left[G(x,t)^2-G(x,t)G_n(x,t)\right]\nonumber\\
&\qquad-G(x,t)^{k-1}\mathbb{E}\left[G(x,t)G_n(x,t)-G_n(x,t)^2\right]\nonumber\\
&\qquad+G(x,t)^{k-1}\mathbb{E}\left[G(x,t)G_n(x,t)-G_n(x,t)^2\right]\nonumber\\
&=G(x,t)^{k-1}\mathbb{E}\left[\left(G(x,t)-G_n(x,t)\right)^2\right]\nonumber\\
& \qquad+G(x,t)^{k-1}\mathbb{E}\left[G(x,t)G_n(x,t)-G_n(x,t)^2\right]\\
&=G(x,t)^{k-1}\mathbb{E}\left[\left(G(x,t)-G_n(x,t)\right)^2\right]\nonumber\\
& \qquad+\mathbb{E}\left[G(x,t)^{k-2}G_n(x,t)\left(G(x,t)^2-G(x,t)G_n(x,t)\right)\right].
\end{align*}
In creating the square term above, we get a similar term to what we started with but with the exponent of $G(x,t)$ decreased by one and the exponent of $G_n(x,t)$ increased by one. We repeat this process until the exponent of $G_n(x,t)$ is $k-1$. We then subtract and add back in the same terms but with $x$ replaced by 1. Specifically,
\begin{align*}
S_k&=\frac{\alpha(k)}{k!}\mathbb{E}\bigg[G(x,t)^{k-1}(G(x,t)-G_n(x,t))^2+G(x,t)^{k-2}G_n(x,t)(G(x,t)-G_n(x,t)^2)\nonumber\\
&\quad\quad\quad\quad+\ldots+G(x,t)G_n(x,t)^{k-2}(G(x,t)-G_n(x,t))^2+G_n(x,t)^{k-1}(G(x,t)-G_n(x,t))^2\bigg]\nonumber\\
&\quad\quad-\frac{\alpha(k)}{k!}\mathbb{E}\bigg[G(1,t)^{k-1}(G(x,t)-G_n(x,t))^2+G(1,t)^{k-2}G_n(1,t)(G(x,t)-G_n(x,t)^2)\nonumber\\
&\qquad\qquad+\ldots+G(1,t)G_n(1,t)^{k-2}(G(x,t)-G_n(x,t))^2+G_n(1,t)^{k-1}(G(x,t)-G_n(x,t))^2\bigg]\nonumber\\
&\quad\quad+\frac{\alpha(k)}{k!}\mathbb{E}\bigg[G(1,t)^{k-1}(G(x,t)-G_n(x,t))^2+G(1,t)^{k-2}G_n(1,t)(G(x,t)-G_n(x,t))^2\nonumber\\
&\qquad\qquad+\ldots+G(1,t)G_n(1,t)^{k-2}(G(x,t)-G_n(x,t))^2+G_n(1,t)^{k-1}(G(x,t)-G_n(x,t))^2\bigg]\nonumber\\
&\quad\quad-\frac{\alpha(k)}{k!}\bigg(kG(1,t)^{k-1}G(x,t)^2-kG(1,t)^{k-1}G(x,t)\mathbb{E}[G_n(x,t)]\nonumber\\
&\qquad\qquad+k\mathbb{E}\left[G_n(1,t)^{k-1}G_n(x,t)^2\right]-kG(x,t)\mathbb{E}\left[G_n(1,t)^{k-1}G_n(x,t)\right]\bigg),
\end{align*}
noting that the first four lines combined give something negative. As $G(x,t)\leq G(1,t)$ and $G_n(x,t)\leq G_n(1,t)$, we have
\begin{align*}
S_k&\leq\frac{\alpha(k)}{k!}\mathbb{E}\left[G(1,t)^{k-1}(G(x,t)^2-2G(x,t)G_n(x,t)+G_n(x,t)^2)\right]\nonumber\\
&\quad\quad+\frac{\alpha(k)}{k!}\mathbb{E}\left[G(1,t)^{k-2}G_n(1,t)(G(x,t)^2-2G(x,t)G_n(x,t)+G_n(x,t)^2)\right]\nonumber\\
&\quad\quad+\ldots+\frac{\alpha(k)}{k!}\mathbb{E}\left[G(1,t)G_n(1,t)^{k-2}(G(x,t)^2-2G(x,t)G_n(x,t)+G_n(x,t)^2)\right]\nonumber\\
&\quad\quad+\frac{\alpha(k)}{k!}\mathbb{E}\left[G_n(1,t)^{k-1}(G(x,t)^2-2G(x,t)G_n(x,t)+G_n(x,t)^2)\right]\nonumber\\
&\quad\quad-\frac{\alpha(k)}{k!}\bigg(kG(1,t)^{k-1}G(x,t)^2-kG(1,t)^{k-1}G(x,t)\mathbb{E}[G_n(x,t)]\nonumber\\
&\qquad\qquad+k\mathbb{E}\left[G_n(1,t)^{k-1}G_n(x,t)^2\right]-kG(x,t)\mathbb{E}\left[G_n(1,t)^{k-1}G_n(x,t)\right]\bigg).
\end{align*}
We gather terms in $G(x,t)^2$, $G_n(x,t)^2$ and $G(x,t)G_n(x,t)$, each multiplied by $|G(1,t)^j-G_n(1,t)^j|$ for some $j$. Here we have taken absolute values in order to not worry about signs. We do this by matching each of the terms from the first four lines of the above with the corresponding term from the last two lines. Doing this we see that
\begin{align*}
S_k&\leq\frac{\alpha(k)}{k!}G(x,t)^2\mathbb{E}\Big[G(1,t)^{k-2}\left|G(1,t)-G_n(1,t)\right|+\ldots+\left|G(1,t)^{k-1}-G_n(1,t)^{k-1}\right|\Big]\nonumber\\
&\quad\quad+\frac{\alpha(k)}{k!}\mathbb{E}\Big[G_n(x,t)^2\Big(\left|G(1,t)^{k-1}-G_n(1,t)^{k-1}\right|+\ldots+G_n(1,t)^{k-2}\left|G(1,t)-G_n(1,t)\right|\Big)\Big]\nonumber\\
&\quad\quad+\frac{\alpha(k)}{k!}G(x,t)\mathbb{E}\Big[G_n(x,t)\Big(G(1,t)^{k-2}\left|G(1,t)-G_n(1,t)\right|+\ldots\nonumber\\
&\qquad\qquad\qquad\qquad\qquad\qquad\qquad+G(1,t)\left|G(1,t)^{k-2}-G_n(1,t)^{k-2}\right|\Big)\Big]\nonumber\\
&\quad\quad+\frac{\alpha(k)}{k!}G(x,t)\mathbb{E}\Big[G_n(x,t)\Big(G_n(1,t)\left|G(1,t)^{k-2}-G_n(1,t)^{k-2}\right|+\ldots\nonumber\\
&\qquad\qquad\qquad\qquad\qquad\qquad\qquad+G_n(1,t)^{k-2}\left|G(1,t)-G_n(1,t)\right|\Big)\Big].
\end{align*}
Hence, using the fact that $G(x,t)$, $G(1,t)$, $G_n(x,t)$, $G_n(1,t)\leq 1$, we see that
\begin{align*}
S_k&\leq \frac{\alpha(k)}{k!}\mathbb{E}\Big[\left|G(1,t)^{k-1}-G_n(1,t)^{k-1}\right|+\ldots+\left|G(1,t)-G_n(1,t)\right|\Big]\nonumber\\
&\quad\quad+\frac{\alpha(k)}{k!}\mathbb{E}\left[\left(\left|G(1,t)^{k-1}-G_n(1,t)^{k-1}\right|+\ldots+\left|G(1,t)-G_n(1,t)\right|\right)\right]\nonumber\\
&\quad\quad+\frac{\alpha(k)}{k!}\mathbb{E}\left[\left|G(1,t)^{k-2}-G_n(1,t)^{k-2}\right|+\ldots+\left|G(1,t)-G_n(1,t)\right|\right]\nonumber\\
&\quad\quad+\frac{\alpha(k)}{k!}\mathbb{E}\left[\left|G(1,t)^{k-2}-G_n(1,t)^{k-2}\right|+\ldots+\left|G(1,t)-G_n(1,t)\right|\right]\\
&=2\frac{\alpha(k)}{k!}\sum_{j=1}^{k-1}\mathbb{E}\left[\left|G(1,t)^j-G_n(1,t)^j\right|\right]+2\frac{\alpha(k)}{k!}\sum_{j=1}^{k-2}\mathbb{E}\left[\left|G(1,t)^j-G_n(1,t)^j\right|\right]\\
S_k&\leq2\frac{\alpha(k)}{k!}\sum_{j=1}^{k-1}j\mathbb{E}\left[\left|G(1,t)-G_n(1,t)\right|\right]+2\frac{\alpha(k)}{k!}\sum_{j=1}^{k-2}j\mathbb{E}\left[\left|G(1,t)-G_n(1,t)\right|\right],
\end{align*}
where the above inequality uses $|y^j-z^j|\leq j|y-z|$ for $y,z\leq1$. Thus
\[
S_k\leq\frac{2\alpha(k)(k-1)}{k(k-2)!}\mathbb{E}\left[\left|G(1,t)-G_n(1,t)\right|\right]
\]
We can plug this result back into \eqref{Y2} and see that 
\begin{align*}
 \frac{d}{dt}Y_{n,x}(t)&\leq-2\lambda Y_{n,x}(t) +4\mathbb{E}\left[\left|G(1,t)-G_n(1,t)\right|\right]\sum_{k=2}^n\frac{\alpha(k)(k-1)}{k(k-2)!} + H_n.
 \end{align*}
Now using lemma \ref{lemma:X}, we have that
\begin{equation}
\frac{d}{dt}Y_{n,x}(t)\leq -2\lambda Y_{n,x}(t)+4\frac{\sqrt{H_n}}{\sqrt{\lambda}}\sum_{k=2}^n\frac{\alpha(k)(k-1)}{k(k-2)!}+H_n.\label{Ychange2}
\end{equation}
Then, we can use Gronwall's Inequality to see that
\begin{align}
Y_{n,x}(t)&\leq\left(2\frac{\sqrt{H_n}}{\lambda \sqrt{\lambda}}\sum_{k=2}^n\frac{\alpha(k)(k-1)}{k(k-2)!}+H_n\right)\left(1-e^{-2\lambda t}\right)\nonumber\\
&\leq2\frac{\sqrt{H_n}}{\lambda \sqrt{\lambda}}\sum_{k=2}^n\frac{\alpha(k)(k-1)}{k(k-2)!}+H_n\nonumber\\
&\to 0\nonumber
\end{align}
as $n\to\infty$, thanks to \eqref{alpha}. This convergence is uniform in both $x$ and $t$, as required.

\end{proof}

From Theorem 1 the Smoluchowski-type equations stated in Corollary 1 follow immendiately: 
\begin{proof}[Proof of Corollary 1]
We begin by subsituting $G(x,t)=\sum_{j}x^j w_j(t)$ into \eqref{hatG};
\begin{align*}
\frac{\partial G}{\partial t}(x,t)&=\lambda(x-G(x,t))+\sum_{k=2}^\infty \frac{\alpha(k)}{k!}(G(x,t)^k-kG(x,t)G(1,t)^{k-1})\\
&=\lambda\left(x-\sum_{j=1}^\infty x^jw_j(t)\right)+\sum_{k=2}^\infty \frac{\alpha(k)}{k!}\left(\left(\sum_{j=1}^\infty x^jw_j(t)\right)^k-k\left(\sum_{j=1}^\infty x^jw_j(t)\right)\left(\sum_{j=1}^\infty w_j(t)\right)^{k-1}\right).
\end{align*}
Rearranging the right hand side to collect powers of $x$ we obtain 
\begin{align*}
\frac{\partial G}{\partial t}(x,t)&=x\left(\lambda(1-w_1(t))-w_1(t)\sum_{k=1}^\infty \frac{\alpha(k)}{(k-1)!}\left(\sum_{j=1}^\infty w_j(t)\right)^{k-1}\right)\\
&\qquad +\sum_{j=2}^\infty x^j\Bigg(-\lambda w_j(t)-w_j(t)\sum_{k=1}^\infty \frac{\alpha(k)}{(k-1)!}\left(\sum_{l=1}^\infty w_l(t)\right)^{k-1}\\
&\qquad\qquad\qquad\qquad\qquad+\sum_{k=2}^j\frac{\alpha(k)}{k!}\!\sum_{\substack{l_1,\ldots,l_k\\ l_1+\cdots+l_k=j}}w_{l_1}(t)\cdots w_{l_k}(t)\Bigg),
\end{align*}
as required, where the final term on the right-hand side is obtained using the product formula for power series.
\end{proof}

\section{Universal Stationary Distribution}
Now that we know how the empirical generating function $G_n(x,t)$ converges in the large $n$ limit, we can extract a prediction for the stationary distribution of cluster sizes. Recalling from \eqref{littleg} that the generating function of the cluster size distribution is obtained via $$g_n(x,t)=\frac{G_n(x,t)}{G_n(1,t)}\,,$$ we reintroduce (and later prove the existence of) the quantities
\begin{align*}
&g_\lambda(x,t)=\frac{G(x,t)}{G(1,t)}, && G_\lambda(x)=\lim_{t\to\infty}G(x,t),\\
&g_\lambda(x)=\lim_{t\to\infty}g_\lambda(x,t)=\frac{G_\lambda(x)}{G_\lambda(1)}, && p_k(\lambda)=\lim_{t\to\infty}\lim_{n\to\infty}p_{n,k}(t).
\end{align*}
Note: as this section focuses on the limit as $\lambda\to0$ of the above quantities, these quantities have been labelled with $\lambda$ to make it easier to follow. Under self-averaging, the stationary cluster size distribution is found simply by determining the fixed point $G_\lambda(x)$ of the differential equation \eqref{hatG} and developing $g_\lambda(x)=G_\lambda(x)/G_\lambda(1)$ as a power series. In practice, this procedure is only tractable in the limit $\lambda\searrow0$, where we find a particular limiting distribution. 

Proof of the convergence of $G$ as $t\to \infty$ is an application of the Poincar\'{e}-Bendixson Theorem (see \cite{Teschl2012} p.~223), which states that attractors of plane autonomous systems are either fixed points, collections of fixed points (with their associated linking orbits), or limit cycles. To compute the asymptotic form of the solution in small $\lambda$ we employ the classical technique of series inversion that dates back to Newton, Lagrange, B\"urmann and Puiseux.In particular, the following theorem is a sub-case of results of those authors (see \cite{EM}, p.~183):
%{\color{black}To prove that $G$ converges as $t$ tends to infinity we use the Poincar\'{e}-Bendixson Theorem (see \cite{Teschl2012} p.~223) stated below

%\begin{theorem}[Poincar\'{e}-Bendixson]\label{PoiBen}
%Given a differentiable real dynamical system defined on an open set of $\mathbb{R}^2$, then suppose that every $\omega$-limit set contains at most one fixed point. Then one of the following holds:
%\begin{itemize}
%\item the $\omega$-limit set is a fixed point,
%\item the $\omega$-limit set is a regular periodic orbit.
%\end{itemize}
%\end{theorem}}
%We employ the classical technique of series inversion that dates back to Newton, Lagrange, B\"urmann and Puiseux. In particular, the following theorem is a sub-case of results of those authors (see \cite{EM}, p.~183):
\begin{theorem}[Series Inversion]\label{series}
Let $f:[0,\infty)\to[0,\infty)$ be defined by the power series $f(x)=\sum_{k\geq m}a_kx^k$ for integer $m\geq 1$ and non-negative real constants $\{a_k\}$. The inverse is expressed via the Puiseux series
$$f^{-1}(y)=\sum_{k=1}^\infty b_k y^{k/m}\,,$$
where
$$b_k=\frac{1}{k!}\lim_{z\to 0}\left[\frac{d^{k-1}}{dz^{k-1}}\left(\frac{z}{f(z)^{1/m}}\right)^k\right]\,.$$
In particular, note that $b_1=1/a_1$.
\end{theorem}

\begin{proof}[Proof of Theorem 2]

Recall that the limit generating function $G(x,t)$ obeys \eqref{hatG}.
Equivalently we may write $\partial G(x,t)/\partial t= F_x(G(x,t),t)$, where $F_x$ is the function defined in \eqref{Fx}. Since $G(1,t)\geq G(x,t)$ for all $x$, it follows that $F_x$ is monotonically decreasing in its first argument, moreover, $F_x(1,t)<0<F_x(0,t)$. Hence, in the particular case $x=1$, $G(1,t)$ has a unique fixed point , denoted $G_\lambda(1)$. As $G(1,t)\in[0,1]$ for all $t$ convergence is assured. 

For general $x$, note that $(G(1,t),G(x,t))$ is a plane autonomous system and, by the same argument as above, has a unique fixed point. From the Poincar\'{e}-Bendixson theorem, we conclude that we need only to rule out the case of periodic orbits to confirm convergence to the fixed point. But we have already shown that $G(1,t)$ converges to its unique fixed point $G_\lambda(1)$, so the $\omega$-limit set of the plane system must be contained in the line $(G_\lambda(1),\cdot)$. Thus, by the Jordan Curve theorem (see \cite{Teschl2012} p.~220), the $\omega$-limit set cannot be a regular periodic orbit. Therefore, we may conclude that the solution of \eqref{hatG} is required to converge to a stationary value $G(x,t)\to G_\lambda(x)$, where
\begin{equation}\label{H}
0=\lambda(x-G_\lambda(x))+\sum_{k=2}^{\infty}\frac{\alpha(k)}{k!}\left(G_\lambda(x)^k-kG_\lambda(1)^{k-1}G_\lambda(x)\right).
\end{equation}
At the point $x=1$ we rearrange to find $\lambda$ as a power series in $G_\lambda(1)$:
\begin{equation}
\lambda=\frac{1}{1-G_\lambda(1)}\sum_{k\geq m}\frac{\alpha(k)}{k!}(k-1)G_\lambda(1)^k=\frac{\alpha(m)}{m!}(m-1)G_\lambda(1)^m+\sum_{k>m}c_kG_\lambda(1)^k\,,
\end{equation}
where $m$ is the least integer such that $\alpha(m)>0$, and $c_k$ are the appropriate constants. From Theorem \ref{series} we thus determine that
\begin{equation*}
  G_\lambda(1)=\lambda^{1/m}\left(\frac{m!}{(m-1)\alpha(m)}\right)^{1/m}+o(\lambda^{1/m}).
\end{equation*}
For general $x$, equation \eqref{H} implies the same leading order behaviour of $G_\lambda(1)$ in $\lambda$. Recalling the definition $g_\lambda(x)=G_\lambda(x)/G_\lambda(1)$ at the start of this section, we see that
\begin{align*}
  0&=\lambda\left(x-G_\lambda(1)g_\lambda(x)\right)+\sum_{k=m}^{\infty}\frac{\alpha(k)}{k!}G_\lambda(1)^k\left(g_\lambda(x)^k-kg_\lambda(x)\right),\\
&=\lambda x - \left(\frac{m!}{(m-1)\alpha(m)}\right)^{1/m}\lambda^{1+1/m}g_\lambda(x)\\
&\qquad+\sum_{k=m}^{\infty}\frac{\alpha(k)}{k!}\left(\frac{m!}{(m-1)\alpha(m)}\right)^{k/m}\lambda^{k/m}\left(g_\lambda(x)^k-kg_\lambda(x)\right) + o(\lambda),\\
&=\lambda x +\frac{\lambda}{m-1}\left(g_\lambda(x)^m-mg_\lambda(x)\right)+o(\lambda), \quad x\in[0,1].
\end{align*}
In other words, when $m$ is the smallest integer such that $\alpha$ is non-zero, we find that in the limit as $\lambda\searrow0$,
\begin{equation*}
mg_\lambda(x)-g_\lambda(x)^{m}- x(m-1)=0.
\end{equation*}
{\color{black}Note that this limit exists as $g_\lambda$ is a Puiseux series in $\lambda$.}
Applying the inversion theorem once more, we obtain the series
\begin{align}
g_\lambda(x)&=\sum_{k=1}^{\infty}\frac{x^k}{k!}\lim_{z\to0}\left[\frac{d^{k-1}}{dz^{k-1}}\left(\frac{z}{f(z)}\right)^k\right]\label{f},
\end{align}
where $f(x)=m/(m-1)z-1/(m-1)z^{m}.$ We compute
\begin{align*}
\left(\frac{z}{f(z)}\right)^k&=\left(\frac{m-1}{m}\right)^k\left(\frac{1}{1-z^{m-1}/m}\right)^k\\
&=\left(\frac{m-1}{m}\right)^k\left(\sum_{n=0}^{\infty}\left(\frac{1}{m}\right)^nz^{n(m-1)}\right)^k\\
&=\left(\frac{m-1}{m}\right)^k\sum_{n=0}^{\infty}\left(\frac{1}{m}\right)^n\binom{k-1+n}{n}z^{n(m-1)},
\end{align*}
using the product formula for series. 
To get the summands for \eqref{f} we need to look at the $(k-1)^{th}$ derivative of the above series, evaluated at $z=0$. Hence we have that
\begin{eqnarray*}
\frac{d^{k-1}}{dz^{k-1}} \left(\frac{z}{f(z)}\right)^k\bigg|_{z=0}=\begin{cases}\left(\frac{m-1}{m}\right)^k\left(\frac{1}{m}\right)^{\frac{k-1}{m-1}}\left(\bfrac{m\left(\frac{k-1}{m-1}\right)}{\frac{k-1}{m-1}}\right)(k-1)! &\mbox{if } m-1\text{ divides }k-1\\
0 &\mbox{otherwise.}\end{cases}
\end{eqnarray*}
Plugging this into \eqref{f} and comparing coefficients of $x^k$, we have
\begin{equation*}
\lim_{\lambda\searrow0}p_k(\lambda)=\begin{cases}\frac{1}{k}\left(\frac{m-1}{m}\right)^k\left(\frac{1}{m}\right)^{\frac{k-1}{m-1}}\left(\bfrac{m\left(\frac{k-1}{m-1}\right)}{\frac{k-1}{m-1}}\right)&\mbox{if } m-1\text{ divides }k-1\\
0 &\mbox{otherwise.}\end{cases}
\end{equation*}
Hence, we can conclude that, for large values of $k$, under these conditions,
\begin{equation*}
\lim_{\lambda\searrow0}p_k(\lambda)\sim k^{-3/2},
\end{equation*}
regardless of the values of $\alpha(k)$, $k\geq2$.
\end{proof}

\section*{Acknowledgements}
We would like to thank Jason Schweinsberg for providing an intuitive explanation for the absence of even-sized clusters in the example of Figure~\ref{fig1}. TR gratefully acknowledges support from the Royal Society.

\bibliographystyle{abbrv}
\bibliography{short,mybibliography}
\end{document}